\documentclass[a4paper]{amsart}
\usepackage{graphicx,amsmath,amsfonts,latexsym,amssymb,amsthm,mathrsfs, color}
\usepackage[latin1]{inputenc}
\renewcommand{\Re}{\mathop{\rm Re}}

\newtheorem{thm}{Theorem}[section]

\newtheorem{corollary}[thm]{Corollary}
\newtheorem{definition}[thm]{Definition}
\newtheorem{theorem}[thm]{Theorem}
\newtheorem{lemma}[thm]{Lemma}
\newtheorem{remark}[thm]{Remark}

\begin{document}
\title[]%
{Preserving closedness of operators under summation}
\author{Nikolaos Roidos}
\address{Institut f\"ur Analysis, Leibniz Universit\"at Hannover, Welfengarten 1, 30167 Hannover, Germany}
\email{roidos@math.uni-hannover.de}

\subjclass[2000]{}
\date{\today}

\begin{abstract}
We give a sufficient condition for the sum of two closed operators to be closed. In particular, we study the sum of two sectorial operators with the sum of their sectoriality angles greater than $\pi$. We show that if one of the operators admits bounded $H^{\infty}$-calculus and the resolvent of the other operator satisfies a boundedness condition stronger than the standard sectoriality, but weaker than the bounded imaginary powers property in the case of UMD spaces, then the sum is closed. We apply the result to the abstract parabolic problem and give a sufficient condition for $L^{p}$-maximal regularity.
\end{abstract}

\date{\today}

\maketitle

\section{introduction}

Let $E$ be a Banach space and $A$, $B$ be two closed linear operators in $E$ with domains $\mathcal{D}(A)$ and $\mathcal{D}(B)$ respectively. We study the problem of whether the sum $A+B$ with domain $\mathcal{D}(A)\cap\mathcal{D}(B)$ is closed. An important application of the closedness of the sum is the $L^{p}$-maximal regularity property. Consider the Cauchy problem
\begin{gather}\label{AP}
\Big\{\begin{array}{lclc} f'(t)+Af(t)=g(t), & t\in(0,\tau)\\
f(0)=0&
\end{array} 
\end{gather}
in the $E$-valued $L^{p}$-space $L^{p}(0,\tau;E)$, where $p>1$ and $\tau>0$ are finite and $-A$ is the infinitesimal generator of a bounded analytic semigroup on $E$. The operator $A$ satisfies {\em $L^{p}$-maximal regularity} if for some $p$ (and hence by \cite{Do} for all) we have that for any $g\in L^{p}(0,\tau;E)$ the unique solution 
\begin{gather*}
f(t)=\int_{0}^{t}e^{(x-t)A}g(x)dx
\end{gather*}
belongs to the first Sobolev space $W^{1,p}(0,\tau;E)$. It is not difficult to see that the above definition is independent of $\tau$. For applications of the $L^{p}$-maximal regularity property to nonlinear problems, using Banach fixed point argument, we refer to \cite{CL}.

In \cite{DV} the abstract problem of the closedness of the sum has been studied in the case of a UMD space (unconditionality of martingale differences) for sectorial operators that have bounded imaginary powers, and a sufficient condition for closedness has been given. In \cite{KW} the same problem has been examined in the approach of operator valued functional calculus, and it has been shown that $R$-sectoriality for one of the operators together with bounded $H^{\infty}$-calculus for the other operator are sufficient for solution. In both cases, an application to the problem of $L^{p}$-maximal regularity has been given. 

In this paper we study the closedness of the sum of two closed operators in the case of sectorial operators with the sum of their sectoriality angles greater than $\pi$. By using the classical formula for the inverse of the closure of the sum of the two operators and the ideas of Theorem 4.4 in \cite{KW} (i.e. dyadic decomposition of $\mathbb{R}$ and an appropriate use of the bounded $H^{\infty}$-calculus property), we observe that the bounded imaginary powers property generates Banach space valued trigonometric polynomials and hence, an analogue boundedness condition can be imposed. In this way we show that if one of the operators has bounded $H^{\infty}$-calculus and the resolvent of the other operator satisfies some Hardy-Littlewood majorant type inequality, which is a boundedness condition stronger than the standard sectoriality and similar to the $R$-sectoriality, then the sum is closed. In the case of spaces having UMD, we further show that the above boundedness condition is weaker than the bounded imaginary powers property. We finally apply the result to the problem (\ref{AP}) and give a sufficient condition for $L^{p}$-maximal regularity in the case of UMD spaces.

\section{The closedness of $A+B$}

The sectoriality property is essential since it implies closability for the sum of two closed operators.

\begin{definition}
Let $E$ be a Banach space, $K\geq1$ and $\theta\in[0,\pi)$. Let $\mathcal{P}_{K}(\theta)$ be the class of closed densely defined linear operators in $E$ such that if $A\in\mathcal{P}_{K}(\theta)$, then 
\[
\Lambda_{\theta}=\{z\in\mathbb{C}\,|\, |\arg z|\leq\theta\}\cup\{0\}\subset\rho{(-A)} \,\,\,\,\,\, \mbox{and} \,\,\,\,\,\, (1+|z|)\|(A+z)^{-1}\|\leq K, \,\,\,\,\,\, \forall z\in \Lambda_{\theta}.
\]
Also, let $\mathcal{P}(\theta)=\cup_{K}\mathcal{P}_{K}(\theta)$. The elements in $\mathcal{P}(\theta)$ are called {\em sectorial operators of angle $\theta$}.
\end{definition}

If $A\in \mathcal{P}_{K}(\theta)$, then by a sectoriality extension argument (see e.g. III.4.7.11 in \cite{Am} or the Appendix of \cite{Ro}) we have that 
\begin{gather*}
\Omega_{K,\theta}=\cup_{\lambda\in\Lambda_{\theta}}\{z\in\mathbb{C}\,|\,|z-\lambda|\leq(1+|\lambda|)/2K\}\subset \rho{(-A)} 
\end{gather*}
and
\begin{gather*}
(1+|z|)\|(A+z)^{-1}\|\leq 2K+1, \,\,\,\,\,\, \forall z\in \Omega_{K,\theta}.
\end{gather*}
For any $\rho\geq0$ and $\theta\in(0,\pi)$, let $\Gamma_{\rho,\theta}$ be the positively oriented path 
\[
\{\rho e^{i\phi}\in\mathbb{C}\,|\,\theta\leq\phi\leq2\pi-\theta\}\cup\{re^{\pm i\theta}\in\mathbb{C}\,|\,r\geq\rho\}.
\]
If $\rho=0$, we denote $\Gamma_{\rho,\theta}$ by $\Gamma_{\theta}$. The complex powers of an operator $A\in\mathcal{P}(\theta)$ in a Banach space $E$ are defined by the Dunford integral. For $\mathrm{Re}(z)<0$ and $\rho>0$ sufficiently small, we have that
\begin{gather*}
A^{z}=\frac{1}{2\pi i}\int_{\Gamma_{\rho,\theta}}(-\lambda)^{z}(A+\lambda)^{-1}d\lambda,
\end{gather*}
which together with $A^{0}=I$ is a strongly continuous holomorphic semigroup on $E$ (see e.g. Theorems III.4.6.2 and III.4.6.5 in \cite{Am}). The imaginary powers are defined by the closure of 
\begin{gather*}
\frac{\sin (i\pi t)}{i\pi t}\int_{0}^{\infty}\lambda^{it}(A+\lambda)^{-2}A d\lambda \,\,\,\,\,\, \mbox{in} \,\,\,\,\,\, \mathcal{D}(A),
\end{gather*}
and can be bounded or unbounded operators. If there exists some $\varepsilon>0$ and $\delta>0$ such that 
\begin{gather*}
A^{it}\in\mathcal{L}(E) \,\,\,\,\,\, \mbox{and} \,\,\,\,\,\, \|A^{it}\|\leq \delta \,\,\,\,\,\, \mbox{for all} \,\,\,\,\,\, t\in[-\varepsilon,\varepsilon],
\end{gather*}
then $A^{it}\in\mathcal{L}(E)$ for all $t\in\mathbb{R}$ and there exist some constants $M\geq 1$ and $\phi\geq0$ such that $\|A^{it}\|\leq M e^{\phi|t|}$, $t\in\mathbb{R}$ (see e.g. Corollary III.4.7.2 in \cite{Am}). In that case we say that {\em $A$ has bounded imaginary powers (with power angle $\phi$)}, and the family $\{A^{z}\, | \, \mathrm{Re}(z)\leq0\}$ is a strongly continuous semigroup on $\mathcal{L}(E)$, i.e. the map from $\{z\in\mathbb{C}\, |\, \mathrm{Re}(z)\leq0\}$ to $\mathcal{L}(E)$ (equipped with the strong operator topology) defined by $z\rightarrow A^{z}$ is a continuous representation of the additive semigroup $\{z\in\mathbb{C}\, |\, \mathrm{Re}(z)\leq0\}$ (this is Theorem III.4.7.1 in \cite{Am}).

\begin{definition} 
Let $E$ be a Banach space and $A\in\mathcal{P}(\theta)$, $\theta\in(0,\pi)$. Let $H_{0}^{\infty}(\theta)$ be the space of all bounded holomorphic functions $f:\mathbb{C}\setminus \Lambda_{\theta}\rightarrow \mathbb{C}$ such that 
\[
|f(\lambda)|\leq c\big(\frac{|\lambda|}{1+|\lambda|^{2}}\big)^{\eta}, \,\,\, \mbox{for any} \,\,\, \lambda\in \mathbb{C}\setminus \Lambda_{\theta}, 
\]
and some $c>0$, $\eta>0$ depending on $f$. Any $f\in H_{0}^{\infty}(\theta)$ can be extended to non-tangential values in $\partial \Lambda_{\theta}$, and defines an element in $\mathcal{L}(E)$ by 
\[
f(-A)=\frac{1}{2\pi i}\int_{\Gamma_{\theta}}f(\lambda)(A+\lambda)^{-1} d\lambda.
\]
We say that the operator $A$ admits a {\em bounded $H^{\infty}$-calculus} if
\begin{gather}\label{eb}
\|f(-A)\|\leq C_{A}\sup_{\lambda\in\mathbb{C}\setminus \Lambda_{\theta}}|f(\lambda)|, \,\,\, \mbox{for any} \,\,\, f\in H_{0}^{\infty}(\theta),
\end{gather}
where $C_{A}>0$ depends only on $A$. 
\end{definition}

Note that operators admitting bounded $H^{\infty}$-calculus automatically have bounded imaginary powers (see e.g. Corollary 2.2 in \cite{CDMY} and Lemma III.4.7.4 in \cite{Am}). Moreover, the boundedness condition (\ref{eb}) also holds for bounded holomorphic functions $f:\mathbb{C}\setminus \Lambda_{\theta}\rightarrow \mathbb{C}$ satisfying 
\begin{gather*}
|f(\lambda)|\leq c\frac{|\lambda|^{\eta}}{1+|\lambda|}, \,\,\, \mbox{for any} \,\,\, \lambda\in \mathbb{C}\setminus \Lambda_{\theta},
\end{gather*}
and some $c>0$, $\eta\in(0,1)$ depending on $f$ (see e.g. Corollary 2.2 in \cite{CDMY}). We recall next a well known decay property of the resolvent.
\begin{lemma}\label{l1}
Let $E$ be a Banach, $A\in\mathcal{P}(\theta)$, $\theta>0$, and $x\in\mathcal{D}(A^{\phi})$ for some $\phi\in(0,1)$. Then, for any $\theta'\in[0,\theta)$ and  $\eta\in[0,\phi)$, $z^{\eta}A(A+z)^{-1}x$ is bounded in $\Lambda_{\theta'}$. 
\end{lemma}
\begin{proof}
There exists a $y\in E$ such that
\begin{gather*}
x=A^{-\phi}y=\frac{1}{2\pi i}\int_{-\delta+\Gamma_{\theta}}(-\lambda)^{-\phi}(A+\lambda)^{-1}yd\lambda,
\end{gather*}
for some $\delta>0$ sufficiently small, where we have used a sectoriality extension argument. Hence, for any $z\in\Lambda_{\theta'}$ by Cauchy's theorem we have that
\begin{eqnarray*}
\lefteqn{z^{\eta}A(A+z)^{-1}x=\frac{z^{\eta}}{2\pi i}A\int_{-\delta+\Gamma_{\theta}}(-\lambda)^{-\phi}(A+z)^{-1}(A+\lambda)^{-1}yd\lambda}\\
&=&\frac{z^{\eta}}{2\pi i}A\int_{-\delta+\Gamma_{\theta}}\frac{(-\lambda)^{-\phi}}{\lambda-z}\big((A+z)^{-1}-(A+\lambda)^{-1}\big)yd\lambda\\
&=&\frac{z^{\eta}}{2\pi i}A(A+z)^{-1}\int_{-\delta+\Gamma_{\theta}}\frac{(-\lambda)^{-\phi}}{\lambda-z}yd\lambda-\frac{z^{\eta}}{2\pi i}A\int_{-\delta+\Gamma_{\theta}}\frac{(-\lambda)^{-\phi}}{\lambda-z}(A+\lambda)^{-1}yd\lambda\\
&=&-\frac{z^{\eta}}{2\pi i}\int_{-\delta+\Gamma_{\theta}}\frac{(-\lambda)^{-\phi}}{\lambda-z}(A+\lambda-\lambda)(A+\lambda)^{-1}yd\lambda\\
&=&-\frac{z^{\eta}}{2\pi i}\int_{-\delta+\Gamma_{\theta}}\frac{(-\lambda)^{-\phi}}{\lambda-z}yd\lambda+\frac{z^{\eta}}{2\pi i}\int_{-\delta+\Gamma_{\theta}}\frac{\lambda(-\lambda)^{-\phi}}{\lambda-z}(A+\lambda)^{-1}yd\lambda\\
&=&\frac{z^{\eta}}{2\pi i}\int_{-\delta+\Gamma_{\theta}}\frac{\lambda(-\lambda)^{-\phi}}{\lambda-z}(A+\lambda)^{-1}yd\lambda.
\end{eqnarray*}
The result now follows by the relation
\begin{gather*}
z^{\eta}A(A+z)^{-1}x=\frac{1}{2\pi i}\int_{\delta-\Gamma_{\theta}}\frac{(\frac{z}{\lambda})^{\eta}}{1+\frac{z}{\lambda}}\lambda^{-(\phi-\eta)}(A-\lambda)^{-1}yd\lambda.
\end{gather*}
\end{proof}

We define next a boundedness condition stronger than the standard sectoriality, that is a Hardy-Littlewood majorant type of estimation defined on Banach space valued trigonometric polynomials. 
\begin{definition}
Let $E$ be a Banach space and $\theta\in[0,\pi)$. Let $\mathcal{T}(\theta)$ be the subclass of $\mathcal{P}(\theta)$ such that if $A\in\mathcal{T}(\theta)$ then for any $\phi\in[-\theta,\theta]$,  $r\in[\frac{1}{e},1]$ and $x_{0},...,x_{n}\in E$, for all $n\in\mathbb{N}$, there exists some collection $a_{0}(t),...,a_{n}(t)\in\{f\in L^{\infty}(0,2\pi)\, | \, \|f\|_{\infty}\leq1\}$ depending on $A$, $n$, $\phi$, $r$ and $x_{0},...,x_{n}$, such that
\[
\|\sum_{k=0}^{n}e^{ikt}(I+re^{-k+i\phi}A)^{-1}x_{k}\|_{L^{p}(0,2\pi;E)} \leq C_{A}^{p,\phi}\|\sum_{k=0}^{n}a_{k}(t)x_{k}\|_{L^{p}(0,2\pi;E)},
\] 
for some fixed $p\in(1,\infty)$ and some constant $C_{A}^{p,\phi}>0$ depending on $A$, $p$ and $\phi$. Let $\mathcal{T}^{\ast}(\theta)$ be the subclass of $\mathcal{T}(\theta)$ such that $a_{0}(t),...,a_{n}(t)$ are independent of  $x_{0},...,x_{n}$. The elements in $\mathcal{T}(\theta)$ and $\mathcal{T}^{\ast}(\theta)$ are called {\em $T$-sectorial} and {\em $T^{\ast}$-sectorial operators of angle $\theta$} respectively. 
\end{definition}

Since we will consider only commuting operators, we recall the following definition. 

\begin{definition}
Two closed linear operators $A$, $B$ in a Banach space $E$ are {\em resolvent commuting} if there exist some $\lambda\in\rho(-A)$ and $\mu\in\rho(-B)$ such that 
\[
[(A+\lambda)^{-1},(B+\mu)^{-1}]=0.
\]
\end{definition}

Then, we can use the commutation properties from Lemmas III.4.9.1 and III.4.9.2 in \cite{Am}. By using the ideas from Theorem 4.4 in \cite{KW}, we have the following result on the closedness of the sum of two sectorial operators.

\begin{theorem}\label{t1}
Let $E$ be a Banach space and $A\in\mathcal{P}(\theta_{A})$, $B\in\mathcal{P}(\theta_{B})$ be resolvent commuting with $\theta_{A}+\theta_{B}>\pi$. If one of the operators is $T$-sectorial and the other one has bounded $H^{\infty}$-calculus, then $A+B$ with domain $\mathcal{D}(A)\cap\mathcal{D}(B)$ is closed and $0\in\rho(A+B)$. 
\end{theorem}
\begin{proof}
We can assume that $\theta_{A}>\theta_{B}$. The fact that $A\in\mathcal{P}(\theta_{A})$ and $B\in\mathcal{P}(\theta_{B})$ with $\theta_{A}+\theta_{B}>\pi$, implies that $A+B$ is closable and its closure has a bounded inverse in $E$ given by 
\begin{gather*}
\mathcal{K}=\overline{(A+B)}^{-1}=\frac{1}{2\pi i}\int_{\Gamma_{\theta_{B}}}(A-z)^{-1}(B+z)^{-1}dz.
\end{gather*}
For details of the above we refer to Theorem 3.7 in \cite{PG} (or to Theorem 2.1 in \cite{Ro}). Also, from the analysis there (see e.g. 2.5 in \cite{Ro}), it follows that a sufficient condition for the closedness of $A+B$ is that $\mathcal{K}$ maps to one of the domains $\mathcal{D}(A)$ or $\mathcal{D}(B)$. By a sectoriality extension argument and Cauchy's theorem, we can replace the path $\Gamma_{\theta_{B}}$ in the definition of $\mathcal{K}$ by $\pm \delta+\Gamma_{\theta_{B}-\varepsilon}$, for some $\delta>0$ and $\varepsilon>0$ sufficiently close to zero. Take $w\in\mathbb{C}$ with $\Re(w)<0$. By Cauchy's theorem and Fubini's theorem (for the Bochner integral), with $\rho>0$ sufficiently small, we have that
\begin{eqnarray*}
\lefteqn{\mathcal{K}A^{w}=\frac{1}{2\pi i}\int_{-\delta+\Gamma_{\theta_{B}}}(A-z)^{-1}(B+z)^{-1}(\frac{1}{2\pi i}\int_{\Gamma_{\rho,\theta_{A}}}(-\lambda)^{w}(A+\lambda)^{-1}d\lambda)dz}\\
&=&(\frac{1}{2\pi i})^{2}\int_{-\delta+\Gamma_{\theta_{B}}}\int_{\Gamma_{\rho,\theta_{A}}}((A-z)^{-1}-(A+\lambda)^{-1})(B+z)^{-1}(\lambda+z)^{-1}(-\lambda)^{w}d\lambda dz\\
&=&(\frac{1}{2\pi i})^{2}\int_{-\delta+\Gamma_{\theta_{B}}}\int_{\Gamma_{\rho,\theta_{A}}}(A-z)^{-1}(B+z)^{-1}(\lambda+z)^{-1}(-\lambda)^{w}d\lambda dz\\
&&-(\frac{1}{2\pi i})^{2}\int_{\Gamma_{\rho,\theta_{A}}}\int_{-\delta+\Gamma_{\theta_{B}}}(A+\lambda)^{-1}(B+z)^{-1}(\lambda+z)^{-1}(-\lambda)^{w}dz d\lambda\\
&=&\frac{1}{2\pi i}\int_{-\Gamma_{\rho,\theta_{A}}}(A-\lambda)^{-1}(B+\lambda)^{-1}\lambda^{w}d\lambda.
\end{eqnarray*}
Since the integral 
\begin{gather*}
\int_{-\Gamma_{\rho,\theta_{A}}}A(A-\lambda)^{-1}(B+\lambda)^{-1}\lambda^{w}d\lambda
\end{gather*}
converges absolutely, we find that $\mathcal{K}A^{w}\in\mathcal{D}(A)$ and
\begin{eqnarray}\nonumber
\lefteqn{A\mathcal{K}A^{w}=\frac{1}{2\pi i}\int_{-\Gamma_{\rho,\theta_{A}}}A(A-\lambda)^{-1}(B+\lambda)^{-1}\lambda^{w}d\lambda}\\\nonumber
&=&\frac{1}{2\pi i}\int_{-\Gamma_{\rho,\theta_{A}}}(A-\lambda+\lambda)(A-\lambda)^{-1}(B+\lambda)^{-1}\lambda^{w}d\lambda\\\nonumber
&=&\frac{1}{2\pi i}\int_{-\Gamma_{\rho,\theta_{A}}}(B+\lambda)^{-1}\lambda^{w}d\lambda+\frac{1}{2\pi i}\int_{-\Gamma_{\rho,\theta_{A}}}(A-\lambda)^{-1}(B+\lambda)^{-1}\lambda^{1+w}d\lambda\\\label{e1}
&=&\frac{1}{2\pi i}\int_{-\Gamma_{\rho,\theta_{A}}}(A-\lambda)^{-1}(B+\lambda)^{-1}\lambda^{1+w}d\lambda.
\end{eqnarray}

Similarly, we have that
\begin{eqnarray*}
\lefteqn{\mathcal{K}B^{w}=\frac{1}{2\pi i}\int_{\delta+\Gamma_{\theta_{B}-\varepsilon}}(A-z)^{-1}(B+z)^{-1}(\frac{1}{2\pi i}\int_{\Gamma_{\rho,\theta_{B}}}(-\lambda)^{w}(B+\lambda)^{-1}d\lambda)dz}\\
&=&(\frac{1}{2\pi i})^{2}\int_{\delta+\Gamma_{\theta_{B}-\varepsilon}}\int_{\Gamma_{\rho,\theta_{B}}}(A-z)^{-1}((B+z)^{-1}-(B+\lambda)^{-1})(\lambda-z)^{-1}(-\lambda)^{w}d\lambda dz\\
&=&(\frac{1}{2\pi i})^{2}\int_{\delta+\Gamma_{\theta_{B}-\varepsilon}}\int_{\Gamma_{\rho,\theta_{B}}}(A-z)^{-1}(B+z)^{-1}(\lambda-z)^{-1}(-\lambda)^{w}d\lambda dz\\
&&-(\frac{1}{2\pi i})^{2}\int_{\Gamma_{\rho,\theta_{B}}}\int_{\delta+\Gamma_{\theta_{B}-\varepsilon}}(A-z)^{-1}(B+\lambda)^{-1}(\lambda-z)^{-1}(-\lambda)^{w}dz d\lambda\\
&=&\frac{1}{2\pi i}\int_{\Gamma_{\rho,\theta_{B}}}(A-\lambda)^{-1}(B+\lambda)^{-1}(-\lambda)^{w}d\lambda.
\end{eqnarray*}
Since the integral
\begin{gather*}
\int_{\Gamma_{\rho,\theta_{B}}}A(A-\lambda)^{-1}(B+\lambda)^{-1}(-\lambda)^{w}d\lambda
\end{gather*}
converges absolutely, we find that $\mathcal{K}B^{w}\in\mathcal{D}(A)$ and
\begin{eqnarray}\nonumber
\lefteqn{A\mathcal{K}B^{w}=\frac{1}{2\pi i}\int_{\Gamma_{\rho,\theta_{B}}}A(A-\lambda)^{-1}(B+\lambda)^{-1}(-\lambda)^{w}d\lambda}\\\nonumber
&=&\frac{1}{2\pi i}\int_{\Gamma_{\rho,\theta_{B}}}(A-\lambda+\lambda)(A-\lambda)^{-1}(B+\lambda)^{-1}(-\lambda)^{w}d\lambda\\\nonumber
&=&\frac{1}{2\pi i}\int_{\Gamma_{\rho,\theta_{B}}}(B+\lambda)^{-1}(-\lambda)^{w}d\lambda-\frac{1}{2\pi i}\int_{\Gamma_{\rho,\theta_{B}}}(A-\lambda)^{-1}(B+\lambda)^{-1}(-\lambda)^{1+w}d\lambda\\\label{e2}
&=&B^{w}-\frac{1}{2\pi i}\int_{\Gamma_{\rho,\theta_{B}}}(A-\lambda)^{-1}(B+\lambda)^{-1}(-\lambda)^{1+w}d\lambda.
\end{eqnarray}

By taking $w=-(\theta+\phi)+it$ with $t\in\mathbb{R}$, $\theta+\phi\in(0,1)$ and $0<\theta,\phi<1$, we can let $\rho=0$ in the equation (\ref{e1}). Then, for any $n\in\mathbb{N}$, we infer that
\begin{eqnarray}\nonumber
\lefteqn{A\mathcal{K}A^{-\theta+it}=A^{\phi}\frac{1}{2\pi i}\int_{-\Gamma_{\theta_{A}}}(A-\lambda)^{-1}(B+\lambda)^{-1}\lambda^{1-(\theta+\phi)+it}d\lambda}\\\nonumber
&=&\frac{1}{2\pi i}\int_{-\Gamma_{\theta_{A}}}A^{\phi}(A-\lambda)^{-1}(B+\lambda)^{-1}\lambda^{1-(\theta+\phi)+it}d\lambda\\\nonumber
&=&\frac{1}{2\pi i}\int_{-\Gamma_{\theta_{A}},|\lambda|\leq 1}A^{\phi}(A-\lambda)^{-1}(B+\lambda)^{-1}\lambda^{1-(\theta+\phi)+it}d\lambda\\\nonumber
&&+\frac{1}{2\pi i}\int_{-\Gamma_{\theta_{A}},1<|\lambda|<e^{n}}A^{\phi}(A-\lambda)^{-1}(B+\lambda)^{-1}\lambda^{1-(\theta+\phi)+it}d\lambda\\\label{e3}
&&+\frac{1}{2\pi i}\int_{-\Gamma_{\theta_{A}},e^{n}\leq |\lambda|}A^{\phi}(A-\lambda)^{-1}(B+\lambda)^{-1}\lambda^{1-(\theta+\phi)+it}d\lambda,
\end{eqnarray}
where we have used the fact that the integral 
\begin{gather*}
\int_{-\Gamma_{\theta_{A}}}A^{\phi}(A-\lambda)^{-1}(B+\lambda)^{-1}\lambda^{1-(\theta+\phi)+it}d\lambda
\end{gather*}
converges absolutely by Lemma \ref{l1} (by noting that $A^{\phi}(A-\lambda)^{-1}=A(A-\lambda)^{-1}A^{\phi-1}$).

Similarly, by (\ref{e2}) with $\rho=0$ we obtain
\begin{eqnarray}\nonumber
\lefteqn{A\mathcal{K}B^{-\theta+it}=B^{-\theta+it}-B^{\phi}\frac{1}{2\pi i}\int_{\Gamma_{\theta_{B}-\varepsilon}}(A-\lambda)^{-1}(B+\lambda)^{-1}(-\lambda)^{1-(\theta+\phi)+it}d\lambda}\\\nonumber
&=&B^{-\theta+it}-\frac{1}{2\pi i}\int_{\Gamma_{\theta_{B}-\varepsilon}}(A-\lambda)^{-1}B^{\phi}(B+\lambda)^{-1}(-\lambda)^{1-(\theta+\phi)+it}d\lambda\\\nonumber
&=&B^{-\theta+it}-\frac{1}{2\pi i}\int_{\Gamma_{\theta_{B}-\varepsilon},|\lambda|\leq 1}(A-\lambda)^{-1}B^{\phi}(B+\lambda)^{-1}(-\lambda)^{1-(\theta+\phi)+it}d\lambda\\\nonumber
&&-\frac{1}{2\pi i}\int_{\Gamma_{\theta_{B}-\varepsilon},1<|\lambda|<e^{n}}(A-\lambda)^{-1}B^{\phi}(B+\lambda)^{-1}(-\lambda)^{1-(\theta+\phi)+it}d\lambda\\\label{e4}
&&-\frac{1}{2\pi i}\int_{\Gamma_{\theta_{B}-\varepsilon},e^{n}\leq |\lambda|}(A-\lambda)^{-1}B^{\phi}(B+\lambda)^{-1}(-\lambda)^{1-(\theta+\phi)+it}d\lambda,
\end{eqnarray}
where we have used the fact that the integral
\begin{gather*}
\int_{\Gamma_{\theta_{B}-\varepsilon}}(A-\lambda)^{-1}B^{\phi}(B+\lambda)^{-1}(-\lambda)^{1-(\theta+\phi)+it}d\lambda
\end{gather*}
converges absolutely by Lemma \ref{l1}.

Assume first that $A$ is $T$-sectorial and $B$ admits a bounded $H^{\infty}$-calculus. By writing $\tilde{\theta}_{B}=\theta_{B}-\varepsilon$ and $\lambda=re^{\pm i\tilde{\theta}_{B}}$ in the third term on the right hand side of (\ref{e4}), we have that
\begin{eqnarray*}
\lefteqn{-\frac{1}{2\pi i}\int_{\Gamma_{\tilde{\theta}_{B}},1<|\lambda|<e^{n}}(A-\lambda)^{-1}B^{\phi}(B+\lambda)^{-1}(-\lambda)^{1-(\theta+\phi)+it}d\lambda}\\
&=&\frac{e^{i\pi(\theta+\phi)}e^{(\pi-\tilde{\theta}_{B})t}}{2\pi i}\int_{1}^{e^{n}}(A-re^{i\tilde{\theta}_{B}})^{-1}B^{\phi}(B+re^{i\tilde{\theta}_{B}})^{-1}r^{1-(\theta+\phi)+it}e^{i\tilde{\theta}_{B}(1-\theta-\phi)}e^{i\tilde{\theta}_{B}}dr\\
&&-\frac{e^{-i\pi(\theta+\phi)}e^{(\tilde{\theta}_{B}-\pi)t}}{2\pi i}\int_{1}^{e^{n}}(A-re^{-i\tilde{\theta}_{B}})^{-1}B^{\phi}(B+re^{-i\tilde{\theta}_{B}})^{-1}r^{1-(\theta+\phi)+it}e^{-i\tilde{\theta}_{B}(1-\theta-\phi)}e^{-i\tilde{\theta}_{B}}dr\\
&=&C_{\theta,\phi}(t)\int_{1}^{e^{n}}(A-re^{i\tilde{\theta}_{B}})^{-1}(r^{-1}B)^{\phi}(r^{-1}B+e^{i\tilde{\theta}_{B}})^{-1}r^{1-\theta+it}e^{i\tilde{\theta}_{B}}\frac{dr}{r}\\
&&-\tilde{C}_{\theta,\phi}(t)\int_{1}^{e^{n}}(A-re^{-i\tilde{\theta}_{B}})^{-1}(r^{-1}B)^{\phi}(r^{-1}B+e^{-i\tilde{\theta}_{B}})^{-1}r^{1-\theta+it}e^{-i\tilde{\theta}_{B}}\frac{dr}{r},
\end{eqnarray*}
where
\begin{gather*}
C_{\theta,\phi}(t)=e^{i\tilde{\theta}_{B}}\frac{e^{i(\pi-\tilde{\theta}_{B})(\theta+\phi)}e^{(\pi-\tilde{\theta}_{B})t}}{2\pi i} \,\,\,\,\,\, \mbox{and} \,\,\,\,\,\, \tilde{C}_{\theta,\phi}(t)=e^{-i\tilde{\theta}_{B}}\frac{e^{i(\tilde{\theta}_{B}-\pi)(\theta+\phi)}e^{(\tilde{\theta}_{B}-\pi)t}}{2\pi i}.
\end{gather*}

If we pass to $e$-adyc decomposition of $[1,e^{n}]$,  for any $u\in E$ we obtain that
\begin{eqnarray*}
\lefteqn{-\frac{1}{2\pi i}\int_{\Gamma_{\tilde{\theta}_{B}},1<|\lambda|<e^{n}}(A-\lambda)^{-1}B^{\phi}(B+\lambda)^{-1}(-\lambda)^{1-(\theta+\phi)+it}ud\lambda}\\
&=&C_{\theta,\phi}(t)\sum_{k=0}^{n-1}\int_{e^{k}}^{e^{k+1}}(A-re^{i\tilde{\theta}_{B}})^{-1}(r^{-1}B)^{\phi}(r^{-1}B+e^{i\tilde{\theta}_{B}})^{-1}r^{1-\theta+it}e^{i\tilde{\theta}_{B}}u\frac{dr}{r}\\
&&-\tilde{C}_{\theta,\phi}(t)\sum_{k=0}^{n-1}\int_{e^{k}}^{e^{k+1}}(A-re^{-i\tilde{\theta}_{B}})^{-1}(r^{-1}B)^{\phi}(r^{-1}B+e^{-i\tilde{\theta}_{B}})^{-1}r^{1-\theta+it}e^{-i\tilde{\theta}_{B}}u\frac{dr}{r}\\
&=&C_{\theta,\phi}(t)\sum_{k=0}^{n-1}\int_{1}^{e}(A-x e^{k}e^{i\tilde{\theta}_{B}})^{-1}(x^{-1}e^{-k}B)^{\phi}(x^{-1}e^{-k}B+e^{i\tilde{\theta}_{B}})^{-1}x^{1-\theta+it}e^{(1-\theta)k}e^{ikt}e^{i\tilde{\theta}_{B}}u\frac{dx}{x}\\
&&-\tilde{C}_{\theta,\phi}(t)\sum_{k=0}^{n-1}\int_{1}^{e}(A-x e^{k}e^{-i\tilde{\theta}_{B}})^{-1}(x^{-1}e^{-k}B)^{\phi}(x^{-1}e^{-k}B+e^{-i\tilde{\theta}_{B}})^{-1}x^{1-\theta+it}e^{(1-\theta)k}e^{ikt}e^{-i\tilde{\theta}_{B}}u\frac{dx}{x}\\
&=&C_{\theta,\phi}(t)\int_{1}^{e}\sum_{k=0}^{n-1}(A-x e^{k}e^{i\tilde{\theta}_{B}})^{-1}\mathcal{B}_{\phi,k}^{+}(x)x^{1-\theta+it}e^{(1-\theta)k}e^{ikt}e^{i\tilde{\theta}_{B}}u\frac{dx}{x}\\
&&-\tilde{C}_{\theta,\phi}(t)\int_{1}^{e}\sum_{k=0}^{n-1}(A-x e^{k}e^{-i\tilde{\theta}_{B}})^{-1}\mathcal{B}_{\phi,k}^{-}(x)x^{1-\theta+it}e^{(1-\theta)k}e^{ikt}e^{-i\tilde{\theta}_{B}}u\frac{dx}{x},
\end{eqnarray*}
where 
\begin{gather*}
\mathcal{B}_{\phi,k}^{\pm}(x)=(x^{-1}e^{-k}B)^{\phi}(x^{-1}e^{-k}B+e^{\pm i\tilde{\theta}_{B}})^{-1}.
\end{gather*}
If we take the $L^{p}$ norm, with the same $p$ as in the $T$-sectoriality of $A$, by H\"older's inequality and Fubini's theorem we find that
\begin{eqnarray*}
\lefteqn{\|-\frac{1}{2\pi i}\int_{\Gamma_{\tilde{\theta}_{B}},1<|\lambda|<e^{n}}(A-\lambda)^{-1}B^{\phi}(B+\lambda)^{-1}(-\lambda)^{1-(\theta+\phi)+it}ud\lambda\|_{L^{p}(0,2\pi;E)}}\\
&\leq&\frac{e^{2\pi(\pi-\tilde{\theta}_{B})}}{2\pi}(\int_{0}^{2\pi}(\int_{1}^{e}\|\sum_{k=0}^{n-1}(A-x e^{k}e^{i\tilde{\theta}_{B}})^{-1}\mathcal{B}_{\phi,k}^{+}(x)x^{1-\theta}e^{(1-\theta)k}e^{ikt}e^{i\tilde{\theta}_{B}}u\|\frac{dx}{x})^{p}dt)^{\frac{1}{p}}\\
&&+\frac{1}{2\pi}(\int_{0}^{2\pi}(\int_{1}^{e}\|\sum_{k=0}^{n-1}(A-x e^{k}e^{-i\tilde{\theta}_{B}})^{-1}\mathcal{B}_{\phi,k}^{-}(x)x^{1-\theta}e^{(1-\theta)k}e^{ikt}e^{-i\tilde{\theta}_{B}}u\|\frac{dx}{x})^{p}dt)^{\frac{1}{p}}\\
&\leq&\frac{e^{2\pi(\pi-\tilde{\theta}_{B})}}{2\pi}(e-1)^{1-\frac{1}{p}}(\int_{0}^{2\pi}\int_{1}^{e}\|\sum_{k=0}^{n-1}(A-x e^{k}e^{i\tilde{\theta}_{B}})^{-1}\mathcal{B}_{\phi,k}^{+}(x)x^{1-\theta}e^{(1-\theta)k}e^{ikt}e^{i\tilde{\theta}_{B}}u\|^{p}\frac{dx}{x^{p}}dt)^{\frac{1}{p}}\\
&&+\frac{1}{2\pi}(e-1)^{1-\frac{1}{p}}(\int_{0}^{2\pi}\int_{1}^{e}\|\sum_{k=0}^{n-1}(A-x e^{k}e^{-i\tilde{\theta}_{B}})^{-1}\mathcal{B}_{\phi,k}^{-}(x)x^{1-\theta}e^{(1-\theta)k}e^{ikt}e^{-i\tilde{\theta}_{B}}u\|^{p}\frac{dx}{x^{p}}dt)^{\frac{1}{p}}\\
&\leq&\frac{e^{2\pi(\pi-\tilde{\theta}_{B})}}{2\pi}(e-1)^{1-\frac{1}{p}}(\int_{1}^{e}\int_{0}^{2\pi}\|\sum_{k=0}^{n-1}(A-x e^{k}e^{i\tilde{\theta}_{B}})^{-1}\mathcal{B}_{\phi,k}^{+}(x)x^{1-\theta}e^{(1-\theta)k}e^{ikt}e^{i\tilde{\theta}_{B}}u\|^{p}dt\frac{dx}{x^{p}})^{\frac{1}{p}}\\
&&+\frac{1}{2\pi}(e-1)^{1-\frac{1}{p}}(\int_{1}^{e}\int_{0}^{2\pi}\|\sum_{k=0}^{n-1}(A-x e^{k}e^{-i\tilde{\theta}_{B}})^{-1}\mathcal{B}_{\phi,k}^{-}(x)x^{1-\theta}e^{(1-\theta)k}e^{ikt}e^{-i\tilde{\theta}_{B}}u\|^{p}dt\frac{dx}{x^{p}})^{\frac{1}{p}}.
\end{eqnarray*}
Hence, by the assumptions on $A$, there exists some collection 
\begin{gather*}
a_{0}(t),...,a_{n-1}(t),\tilde{a}_{0}(t),...,\tilde{a}_{n-1}(t)\in\{f\in L^{\infty}(0,2\pi)\, | \, \|f\|_{\infty}\leq1\}
\end{gather*}
(where in the following we use the same symbol to denote some appropriate representatives of the elements of he above collection) and some constants $C_{A}^{p,\pi-\tilde{\theta}_{B}}$, $C_{A}^{p,\tilde{\theta}_{B}-\pi}$ such that
\begin{eqnarray*}
\lefteqn{\|-\frac{1}{2\pi i}\int_{\Gamma_{\tilde{\theta}_{B}},1<|\lambda|<e^{n}}(A-\lambda)^{-1}B^{\phi}(B+\lambda)^{-1}(-\lambda)^{1-(\theta+\phi)+it}ud\lambda\|_{L^{p}(0,2\pi;E)}}\\
&\leq&\frac{e-1}{2\pi}e^{2\pi(\pi-\tilde{\theta}_{B})}\sup_{x\in(1,e)}(\int_{0}^{2\pi}\|\sum_{k=0}^{n-1}(A-x e^{k}e^{i\tilde{\theta}_{B}})^{-1}\mathcal{B}_{\phi,k}^{+}(x)x^{1-\theta}e^{(1-\theta)k}e^{ikt}e^{i\tilde{\theta}_{B}}u\|^{p}dt)^{\frac{1}{p}}\\
&&+\frac{e-1}{2\pi}\sup_{x\in(1,e)}(\int_{0}^{2\pi}\|\sum_{k=0}^{n-1}(A-x e^{k}e^{-i\tilde{\theta}_{B}})^{-1}\mathcal{B}_{\phi,k}^{-}(x)x^{1-\theta}e^{(1-\theta)k}e^{ikt}e^{-i\tilde{\theta}_{B}}u\|^{p}dt)^{\frac{1}{p}}\\
&\leq&C_{A}^{p,\pi-\tilde{\theta}_{B}}e^{2\pi(\pi-\tilde{\theta}_{B})}\frac{e-1}{2\pi}\sup_{x\in(1,e)}\sup_{a_{k}}(\int_{0}^{2\pi}\|\sum_{k=0}^{n-1}\mathcal{B}_{\phi,k}^{+}(x)x^{-\theta}e^{-\theta k}a_{k}(t)u\|^{p}dt)^{\frac{1}{p}}\\
&&+C_{A}^{p,\tilde{\theta}_{B}-\pi}\frac{e-1}{2\pi}\sup_{x\in(1,e)}\sup_{\tilde{a}_{k}}(\int_{0}^{2\pi}\|\sum_{k=0}^{n-1}\mathcal{B}_{\phi,k}^{-}(x)x^{-\theta}e^{-\theta k}\tilde{a}_{k}(t)u\|^{p}dt)^{\frac{1}{p}}\\
&\leq&C_{A}^{p,\pi-\tilde{\theta}_{B}}\frac{(e-1)}{(2\pi)^{1-\frac{1}{p}}}e^{2\pi(\pi-\tilde{\theta}_{B})}\sup_{x\in(1,e)}\sup_{a_{k}}\sup_{t\in(0,2\pi)}\|\sum_{k=0}^{n-1}\mathcal{B}_{\phi,k}^{+}(x)x^{-\theta}e^{-\theta k}a_{k}(t)u\|\\
&&+C_{A}^{p,\tilde{\theta}_{B}-\pi}\frac{(e-1)}{(2\pi)^{1-\frac{1}{p}}}\sup_{x\in(1,e)}\sup_{\tilde{a}_{k}}\sup_{t\in(0,2\pi)}\|\sum_{k=0}^{n-1}\mathcal{B}_{\phi,k}^{-}(x)x^{-\theta}e^{-\theta k}\tilde{a}_{k}(t)u\|.
\end{eqnarray*}
Thus, by the assumption on $B$, there exists some constant $C_{B}$ such that
\begin{eqnarray}\nonumber
\lefteqn{\|-\frac{1}{2\pi i}\int_{\Gamma_{\tilde{\theta}_{B}},1<|\lambda|<e^{n}}(A-\lambda)^{-1}B^{\phi}(B+\lambda)^{-1}(-\lambda)^{1-(\theta+\phi)+it}ud\lambda\|_{L^{p}(0,2\pi;E)}}\\\nonumber
&\leq&C_{A}^{p,\pi-\tilde{\theta}_{B}}C_{B}\frac{(e-1)}{(2\pi)^{1-\frac{1}{p}}}e^{2\pi(\pi-\tilde{\theta}_{B})}\sup_{x\in(1,e)}\sup_{a_{k}}\sup_{t\in(0,2\pi)}\sup_{z\in\mathbb{C}\setminus\Lambda_{\theta_{B}}}|\sum_{k=0}^{n-1}\frac{(-x^{-1}e^{-k}z)^{\phi}}{-x^{-1}e^{-k}z+e^{i\tilde{\theta}_{B}}}x^{-\theta}e^{-\theta k}a_{k}(t)|\|u\|\\\nonumber
&&+C_{A}^{p,\tilde{\theta}_{B}-\pi}C_{B}\frac{(e-1)}{(2\pi)^{1-\frac{1}{p}}}\sup_{x\in(1,e)}\sup_{\tilde{a}_{k}}\sup_{t\in(0,2\pi)}\sup_{z\in\mathbb{C}\setminus\Lambda_{\theta_{B}}}|\sum_{k=0}^{n-1}\frac{(-x^{-1}e^{-k}z)^{\phi}}{-x^{-1}e^{-k}z+e^{-i\tilde{\theta}_{B}}}x^{-\theta}e^{-\theta k}\tilde{a}_{k}(t)|\|u\|\\\nonumber
&\leq&C_{A}^{p,\pi-\tilde{\theta}_{B}}C_{B}\frac{(e-1)}{(2\pi)^{1-\frac{1}{p}}}e^{2\pi(\pi-\tilde{\theta}_{B})}\sup_{x\in(1,e)}\sup_{z\in\mathbb{C}\setminus\Lambda_{\theta_{B}}}\sum_{k=0}^{\infty}|\frac{(-x^{-1}e^{-k}z)^{\phi}}{-x^{-1}e^{-k}z+e^{i\tilde{\theta}_{B}}}|\|u\|\\\label{669}
&&+C_{A}^{p,\tilde{\theta}_{B}-\pi}C_{B}\frac{(e-1)}{(2\pi)^{1-\frac{1}{p}}}\sup_{x\in(1,e)}\sup_{z\in\mathbb{C}\setminus\Lambda_{\theta_{B}}}\sum_{k=0}^{\infty}|\frac{(-x^{-1}e^{-k}z)^{\phi}}{-x^{-1}e^{-k}z+e^{-i\tilde{\theta}_{B}}}|\|u\|,
\end{eqnarray}
where the last term is finite. By applying $B^{-it}$ in \eqref{e4} and then taking the $L^{p}$ norm we obtain that
\begin{eqnarray*}
\lefteqn{(2\pi)^{\frac{1}{p}}\|A\mathcal{K}B^{-\theta}u\|}\\
&\leq&(2\pi)^{\frac{1}{p}}\|B^{-\theta}u\|+\|\frac{1}{2\pi i}B^{-it}\int_{\Gamma_{\tilde{\theta}_{B}},|\lambda|\leq 1}(A-\lambda)^{-1}B^{\phi}(B+\lambda)^{-1}(-\lambda)^{1-(\theta+\phi)+it}ud\lambda\|_{L^{p}(0,2\pi;E)}\\
&&+(\sup_{t\in(0,2\pi)}\|B^{-it}\|)\|-\frac{1}{2\pi i}\int_{\Gamma_{\tilde{\theta}_{B}},1<|\lambda|<e^{n}}(A-\lambda)^{-1}B^{\phi}(B+\lambda)^{-1}(-\lambda)^{1-(\theta+\phi)+it}ud\lambda\|_{L^{p}(0,2\pi;E)}\\\
&&+\|\frac{1}{2\pi i}B^{-it}\int_{\Gamma_{\tilde{\theta}_{B}},e^{n}\leq |\lambda|}(A-\lambda)^{-1}B^{\phi}(B+\lambda)^{-1}(-\lambda)^{1-(\theta+\phi)+it}ud\lambda\|_{L^{p}(0,2\pi;E)}.
\end{eqnarray*}
By taking the limit in the above inequality as $n\rightarrow\infty$, we find by \eqref{669} that there exists some constant $C_{A,B}>0$ independent of $\theta$ such that 
\begin{gather*}
\|A\mathcal{K}B^{-\theta}u\| \leq C_{A,B}\|u\|.
\end{gather*}
Since $\mathcal{K}:\mathcal{D}(A)\rightarrow\mathcal{D}(A)$ and $B^{-\theta}A\mathcal{K}\subset A\mathcal{K}B^{-\theta}$, by taking $\theta\rightarrow 0$ in the above relation we find that 
\begin{gather*}
\|A\mathcal{K}v\| \leq C_{A,B}\|v\| \,\,\,\,\,\, \mbox{for any} \,\,\,\,\,\, v\in \mathcal{D}(A).
\end{gather*}
The result then follows by a Cauchy sequence argument and the closedness of $A$. The case when $A$ admits a bounded $H^{\infty}$-calculus and $B$ is $T$-sectorial can be treated in a similarly way starting from (\ref{e3}).
\end{proof}

\begin{remark}
Theorem \ref{t1} still holds if in the definition of $T$-sectoriality we let $p$ to be equal to one or if we replace $e^{ikt}$ with $e^{imkt}$, for all $k$ and some fixed $m\in\mathbb{N}$ (the last follows by using $e^{m}$-adyc decomposition in the proof). Moreover, the same approach can be applied to the general case of the boundedness of the operator valued factional calculus, as in \cite{KW}. 
\end{remark}

We show next that the bounded imaginary powers property is stronger than the $T$-sectoriality, in the case of a UMD space. 

\begin{theorem}\label{t2}
Let $E$ be a UMD Banach space and $A$ be a sectorial operator in $E$ having bounded imaginary powers with power angle $\phi<\pi$. Then $A$ is $T$-sectorial of angle $\theta$ for any $\theta\in[0,\pi-\phi)$.
\end{theorem}
\begin{proof}
We use the representation formulas of the resolvent from the proof of Theorem 4 in \cite{CP}, and follow similar steps. Namely, if $A$ has bounded imaginary powers with power angle $\phi<\pi$, then 
\begin{gather*}
(I+\rho A)^{-1}x=\frac{1}{2\pi i} PV\int_{\mathbb{R}}(\rho A)^{-is}\frac{\pi}{\sinh(\pi s)}xds +\frac{1}{2}x, \,\,\,\,\,\, \mbox{for any} \,\,\,\,\,\, x\in E \,\,\,\,\,\, \mbox{and} \,\,\,\,\,\, \rho>0,
\end{gather*}
and
\begin{gather*}
(I+\rho e^{i\theta}A)^{-1}=(I+\rho A)^{-1}+\frac{1}{2\pi i}\int_{\mathbb{R}}(\rho A)^{-is}\frac{\pi(e^{\theta s}-1)}{\sinh(\pi s)}ds, \,\,\,\,\,\, \mbox{for any} \,\,\,\,\,\, |\theta|<\pi-\phi \,\,\,\,\,\, \mbox{and} \,\,\,\,\,\, \rho>0, 
\end{gather*}
where by $PV$ we mean Cauchy's principal value. For any $p\in[1,\infty)$, $\theta\in(\phi-\pi,\pi-\phi)$, $r\in[\frac{1}{e},1]$ and $x_{0},...,x_{n}\in E$, for any $n\in\mathbb{N}$, we have that
\begin{eqnarray}\nonumber
\lefteqn{\|\sum_{k=0}^{n}e^{ikt}(I+re^{-k+i\theta}A)^{-1}x_{k}\|_{L^{p}(0,2\pi;E)}}\\\nonumber
&\leq&\|\sum_{k=0}^{n}e^{ikt}(I+re^{-k}A)^{-1}x_{k}\|_{L^{p}(0,2\pi;E)}+\|\frac{1}{2\pi i}\int_{\mathbb{R}}\sum_{k=0}^{n}e^{ikt}(re^{-k} A)^{-is}\frac{\pi(e^{\theta s}-1)}{\sinh(\pi s)}x_{k}ds\|_{L^{p}(0,2\pi;E)}\\\nonumber
&\leq&\|\frac{1}{2\pi i} \int_{\mathbb{R}}\sum_{k=0}^{n}e^{ikt}(re^{-k}A)^{-is}(\frac{\pi}{\sinh(\pi s)}-\frac{\chi(s)}{s})x_{k}ds\|_{L^{p}(0,2\pi;E)}\\\nonumber
&&+\|\frac{1}{2\pi i}PV \int_{\mathbb{R}}\sum_{k=0}^{n}e^{ikt}(re^{-k}A)^{-is}\frac{\chi(s)}{s}x_{k}ds\|_{L^{p}(0,2\pi;E)}\\\label{ola}
&&+\|\sum_{k=0}^{n}\frac{e^{ikt}}{2}x_{k}\|_{L^{p}(0,2\pi;E)}+\|\frac{1}{2\pi i}\int_{\mathbb{R}}\sum_{k=0}^{n}e^{ikt}(re^{-k} A)^{-is}\frac{\pi(e^{\theta s}-1)}{\sinh(\pi s)}x_{k}ds\|_{L^{p}(0,2\pi;E)},
\end{eqnarray}
where $\chi(s)$ is the characteristic function of the interval $[-\pi,\pi]$. For the first term on the right hand side of the above equation we estimate
\begin{eqnarray}\nonumber
\lefteqn{\|\frac{1}{2\pi i} \int_{\mathbb{R}}\sum_{k=0}^{n}e^{ikt}(re^{-k}A)^{-is}(\frac{\pi}{\sinh(\pi s)}-\frac{\chi(s)}{s})x_{k}ds\|_{L^{p}(0,2\pi;E)}}\\\nonumber
&\leq&\frac{1}{2\pi}\int_{\mathbb{R}}\|A^{-is}\||\frac{\pi}{\sinh(\pi s)}-\frac{\chi(s)}{s}| \|\sum_{k=0}^{n}e^{ik(s+t)}x_{k}\|_{L^{p}(0,2\pi;E)}ds\\\label{ola1}
&\leq&\frac{1}{2\pi}(\int_{\mathbb{R}}\|A^{-is}\||\frac{\pi}{\sinh(\pi s)}-\frac{\chi(s)}{s}|(1+|s|)ds) \sup_{s\in\mathbb{R}}\|\sum_{k=0}^{n}\frac{e^{ik(s+t)}}{1+|s|}x_{k}\|_{L^{p}(0,2\pi;E)}.
\end{eqnarray}
Similarly, for the last term on the right hand side of (\ref{ola}) we have that 
\begin{eqnarray}\nonumber
\lefteqn{\|\frac{1}{2\pi i}\int_{\mathbb{R}}\sum_{k=0}^{n}e^{ikt}(re^{-k} A)^{-is}\frac{\pi(e^{\theta s}-1)}{\sinh(\pi s)}x_{k}ds\|_{L^{p}(0,2\pi;E)}}\\\nonumber
&\leq&\frac{1}{2\pi}\int_{\mathbb{R}}\|A^{-is}\||\frac{\pi(e^{\theta s}-1)}{\sinh(\pi s)}|\|\sum_{k=0}^{n}e^{ik(s+t)}x_{k}\|_{L^{p}(0,2\pi;E)}ds\\\label{ola2}
&\leq&\frac{1}{2\pi}(\int_{\mathbb{R}}\|A^{-is}\||\frac{\pi(e^{\theta s}-1)}{\sinh(\pi s)}|(1+|s|)ds)\sup_{s\in\mathbb{R}}\|\sum_{k=0}^{n}\frac{e^{ik(s+t)}}{1+|s|}x_{k}\|_{L^{p}(0,2\pi;E)}.
\end{eqnarray}
Finally, we estimate the second term on the right hand side of (\ref{ola}) by using the UMD property of the space $E$, i.e. the boundedness of the Hilbert transform, as follows
\begin{eqnarray}\nonumber
\lefteqn{\|\frac{1}{2\pi i}PV \int_{\mathbb{R}}\sum_{k=0}^{n}e^{ikt}(re^{-k}A)^{-is}\frac{\chi(s)}{s}x_{k}ds\|_{L^{p}(0,2\pi;E)}}\\\nonumber
&=&\|\frac{1}{2\pi i}PV \int_{\mathbb{R}}\sum_{k=0}^{n}(-1)^{k}e^{ikt}(re^{-k}A)^{-is}\frac{\chi(s)}{s}x_{k}ds\|_{L^{p}(-\pi,\pi;E)}\\\nonumber
&=&\|(rA)^{it}(rA)^{-it}\frac{1}{2\pi i}PV \int_{-\pi}^{\pi}\sum_{k=0}^{n}(-1)^{k}\frac{e^{ikt}}{s}(re^{-k}A)^{is}x_{k}ds\|_{L^{p}(-\pi,\pi;E)}\\\nonumber
&\leq&(\sup_{t\in[-\pi,\pi]}\|A^{it}\|)\|\frac{1}{2\pi i}PV \int_{-\pi}^{\pi}\sum_{k=0}^{n}(-1)^{k}\frac{e^{ik(t-s)}}{s}(rA)^{i(s-t)}x_{k}ds\|_{L^{p}(-\pi,\pi;E)}\\\nonumber
&\leq&C(\sup_{t\in[-\pi,\pi]}\|A^{it}\|)\|\sum_{k=0}^{n}(-1)^{k}e^{ikt}(rA)^{-it}x_{k}\|_{L^{p}(-\pi,\pi;E)}\\\label{ola3}
&\leq&C(\sup_{t\in[-\pi,\pi]}\|A^{it}\|)^{2}\|\sum_{k=0}^{n}e^{ikt}x_{k}\|_{L^{p}(0,2\pi;E)},
\end{eqnarray}
for some fixed constant $C$. The result now follows by (\ref{ola}),  (\ref{ola1}),  (\ref{ola2}) and  (\ref{ola3}) (note that the standard sectoriality follows by Theorem 4 in \cite{CP}).
\end{proof}

\section{$L^{p}$-maximal regularity}

Let $E$ be a Banach space and let the operator $B=\partial_{t}$ in $L^{p}(0,\tau;E)$ with 
\begin{gather*}
\mathcal{D}(B)=\{f(t)\in W^{1,p}(0,\tau;E)\,|\,f(0)=0\}, 
\end{gather*}
for some $p\in(1,\infty)$ and $\tau>0$ finite. We have that $\sigma(B)=\emptyset$, and for any $g\in L^{p}(0,\tau;E)$,
\begin{gather*}
(B+\lambda)^{-1}g=\int_{0}^{t}e^{\lambda(x-t)}g(x)dx, \,\,\, \forall \lambda\in\mathbb{C}.
\end{gather*} 
By Young's inequality for convolution (see III.4.2 in \cite{Am}), we infer that
\begin{gather*}
\|(B+\lambda)^{-1}\|\leq\frac{1-e^{-\mathrm{Re}(\lambda) \tau}}{\mathrm{Re}(\lambda)}, \,\,\, \forall \lambda\in\mathbb{C}.
\end{gather*} 
Hence, $B\in\mathcal{P}(\phi)$, for any $\phi\in[0,\pi/2)$. Furthermore, if the space $E$ is UMD, then $B$ admits a bounded $H^{\infty}$-calculus (see e.g. Theorem 8.5.8 in \cite{H}). By substituting $f(t)=e^{ct}h(t)$, $c\in\mathbb{R}$, in (\ref{AP}), we see that in the definition of the $L^{p}$-maximal regularity property we can consider $A\in\mathcal{P}(\theta)$ with $\theta>\frac{\pi}{2}$ instead of $A$ being an infinitesimal generator of a bounded analytic semigroup. Moreover, if $A$ is $T^{\ast}$-sectorial in $E$, by taking $p$ to be the same as that one in the $T^{\ast}$-sectorility of $A$, we see by Fubini's theorem that $A$ can be naturally extended to a $T^{\ast}$-sectorial operator in $L^{p}(0,\tau;E)$ by $(Af)(t)=Af(t)$. Thus, since $L^{p}$-maximal regularity is independent of $p$, Theorem \ref{t1} implies the following.

\begin{theorem}\label{t3}
Let $E$ be a UMD Banach space and $A$ be a sectorial operator in $E$ of angle greater than $\frac{\pi}{2}$. If for some $p\in(1,\infty)$ and $\tau>0$ finite the extension of $A$ in $L^{p}(0,\tau;E)$ is $T$-sectorial of angle greater than $\frac{\pi}{2}$, then $A$ has $L^{p}$-maximal regularity. 
\end{theorem}

\begin{corollary}
In a UMD Banach space any $T^{\ast}$-sectorial operator of angle greater than $\frac{\pi}{2}$ has $L^{p}$-maximal regularity. 
\end{corollary}

\begin{corollary}
If $E$ is a UMD Banach space and $A$ is a sectorial operator in $E$ having bounded imaginary powers with power angle $\phi<\frac{\pi}{2}$, then the extension of $A$ in $L^{p}(0,\tau;E)$ has again bounded imaginary powers with the same power angle $\phi$. Hence, Theorems \ref{t2} and \ref{t3} imply that $A$ has $L^{p}$-maximal regularity, which is the classical result of \cite{DV}.
\end{corollary}

\begin{corollary}
If $E$ is a Hilbert space, by taking $p=2$ we see that $T^{\ast}$-sectoriality becomes equivalent to standard sectoriality (by choosing $a_{k}(t)=e^{ikt}$ for all $k$). Since any Hilbert space is UMD, this proves the result of \cite{D}, i.e. that in a Hilbert space any infinitesimal generator of a bounded analytic semigroup has $L^{p}$-maximal regularity.
\end{corollary}


\begin{thebibliography}{99}

\bibitem{Am} H. Amann, {\em Linear and quasilinear parabolic problems}. Monographs in Mathematics Vol. 89, Birkh\"auser Verlag (1995).
        
\bibitem{CDMY} M. Cowling, I. Doust, A. McIntosh and A. Yagi, {\em Banach space operators with a bounded $H^{\infty}$ functional calculus}. J. Austral. Math. Soc. Ser. A 60, no. 1, 51--89 (1996).

\bibitem{D} L. De Simon, {\em Un' applicazione della teoria degli integrali singolari allo studio delle equazioni differenziali lineari astratte del primo ordine}. Rend. Sem. Mat. Univ. Padova 34 205--223 (1964).

\bibitem{CL} P. Clément and S. Li, {\em Abstract parabolic quasilinear equations and application to a groundwater flow problem}. Adv. Math. Sci. Appl. 3, Special Issue, 17--32 (1993/94).

\bibitem{CP} P. Clément and J. Pr\"uss, {\em An operator-valued transference principle and maximal regularity on vector-valued Lp-spaces}. In: G. Lumer and L. Weis (eds.), Proc. of the 6th. International Conference on Evolution equations. Marcel Dekker (2001).

\bibitem{Do} G. Dore, {\em $L^{p}$ regularity for abstract differential equations} (In ``Functional Analysis and related topics'', editor: H. Komatsu), Lect. Notes in Math. 1540, Springer Verlag (1993). 

\bibitem{DV} G. Dore and A. Venni, {\em On the closedness of the sum of two closed operators}. Math. Z. {\em 196}, 189--201 (1987).

\bibitem{H} M. Haase, {\em The functional calculus for sectorial operators}. Operator theory: Advances and applications, Vol. 169, Birkh\"auser (2006). 

\bibitem{KW} N. Kalton and L. Weis, {\em The $H^{\infty}$-calculus and sums of closed operators}. Math. Ann. 321, no. 2, 319--345 (2001). 

\bibitem{PG} G. Da Prato and P. Grisvard, {\em Sommes d'opérateurs linéaires et équations différentielles opérationnelles}. J. Math. Pures Appl. (9) 54, no. 3, 305--387 (1975).

\bibitem{Ro} N. Roidos, {\em On the inverse of the sum of two sectorial operators}. J. Funct. Anal. 265, no. 2, 208--222 (2013).

\end{thebibliography}
\end{document}